\documentclass[final,1p]%[final,1p,times]%[preprint,12pt]
{elsarticle}

\usepackage{amssymb}
\usepackage{amsthm}

\usepackage{amsmath, amsthm, amssymb} % the order is KEY for \square !!!
\usepackage{graphicx}
\newtheorem{theorem}{Theorem}[section]
\newtheorem{lemma}[theorem]{Lemma}

\newtheorem{corollary}[theorem]{Corollary}

\newtheorem{remark}[theorem]{\it Remark}

\numberwithin{equation}{section}
    \usepackage{color}  %!!!!!!!!!!!!!
     \definecolor{pass}{rgb}{0,0,0.8}%{0.6,0,0.2}     %===provisionally ok
     \definecolor{pass1}{rgb}{0,0,0.5}
\newcommand{\pt}{\partial}

\newcommand{\R}{\mathbb{R}}

\newcommand{\EE}{{\mathcal E}}
\newcommand{\RR}{{\mathcal R}}
\newcommand{\LL}{{\mathcal L}}

\newcommand {\beq} {\begin{equation}}
\newcommand {\eeq} {\end{equation}}
\newcommand {\beqa} {\begin{eqnarray}}
\newcommand {\eeqa} {\end{eqnarray}}
\newcommand {\beqann} {\begin{eqnarray*}}
\newcommand {\eeqann} {\end{eqnarray*}}

\newcommand{\TOL}{{\it TOL}}

\journal{Applied Mathematics Letters}
\begin{document}

\begin{frontmatter}



\title{%A posteriori error estimation
Pointwise-in-time a posteriori error control\\
for time-fractional parabolic equations}


\author{Natalia Kopteva\corref{cor1}}
\cortext[cor1]{Department of Mathematics and Statistics, University of Limerick, Limerick, Ireland}
\ead{natalia.kopteva@ul.ie}
%\ead[url]{https://staff.ul.ie/natalia/ }
%\affiliation{organization={Department of Mathematics and Statistics, University of Limerick},
%Department and Organization
%            addressline={XXX},
%            city={Limerick},
%            postcode={XXX},
%            state={XXX},
%            country={Ireland}}

\begin{abstract}
For
time-fractional parabolic equations with a Caputo time derivative
of order $\alpha\in(0,1)$,
%an initial-boundary value problem with a Caputo time derivative of fractional order $\alpha\in(0,1)$,
we give pointwise-in-time a posteriori error bounds in the spatial $L_2$ and $L_\infty$ norms. Hence,
an adaptive mesh construction algorithm is applied for the L1 method, which yields optimal convergence rates  $2-\alpha$
in the presence of solution singularities.
\end{abstract}



%\begin{keyword}



%\end{keyword}

\end{frontmatter}



%\newpage

\section{Introduction}

%The purpose of this paper is to propose/set a first step in the a posteriori error estimation for
Consider a fractional-order parabolic equation, of order $\alpha\in(0,1)$, of the form
\beq\label{problem}
%\begin{array}{l}
D_t^{\alpha}u+\LL u=f(x,t)\qquad\mbox{for}\;\;(x,t)\in\Omega\times(0,T],%\\[0.2cm]
%u(x,t)=0\quad\mbox{for}\;\;(x,t)\in\pt\Omega\times(0,T],\qquad
%u(x,0)=u_0(x)\quad\mbox{for}\;\;x\in\Omega.
%\end{array}
\eeq
%where $\alpha\in(0,1)$,
subject to the initial condition $u(\cdot,0)=u_0$ in $\Omega$, and the boundary condition $u=0$ on $\pt\Omega$ for $t>0$.
This problem is posed in a bounded Lipschitz domain  $\Omega\subset\R^d$ (where $d\in\{1,2,3\}$), and involves
a spatial linear second-order elliptic operator~$\LL$.
The Caputo fractional derivative in time, denoted here by $D_t^\alpha$, is  defined \cite{Diet10},
for $t>0$, by
\begin{equation}\label{CaputoEquiv}
D_t^{\alpha} u := J_t^{1-\alpha}(\pt_t u),\qquad
J_t^{1-\alpha} v(\cdot,t) :=  \frac1{\Gamma(1-\alpha)} \int_{0}^t(t-s)^{-\alpha}\, v(\cdot, s)\, ds,
    %\quad\text{for }\ t>0,
\end{equation}
where $\Gamma(\cdot)$ is the Gamma function, and $\pt_t$ denotes the partial derivative in~$t$.

Although there is a substantial literature on the a posteriori error estimation for classical parabolic  equations,
the pointwise-in-time a posteriori error control
  appears an open question for equations of type \eqref{problem}
(the few papers for similar problems give  error estimates %in time in weighted/
in global fractional Sobolev space norms %, which cannot detect the solution singularity at $t=0$
\cite{yangpin}).

In this paper, we shall address this question by deriving pointwise-in-time a posteriori error bounds in the
$L_2(\Omega)$ and $L_\infty(\Omega)$ norms.
Furthermore, explicit upper barriers on the residual will be obtained that guarantee
that the error remains within a prescribed tolerance and within certain desirable pointwise-in-time error profiles.
These residual barriers naturally lead to an adaptive mesh construction algorithm,
which will be applied %tested
for the popular L1 method.
%
%will be employed to adaptively construct temporal meshes for the popular L1 method.
It will be demonstrated that the
%adaptive algorithm
constructed adaptive meshes
successfully detect solution singularities
and yield
optimal convergence rates $2-\alpha$, with
the error profiles in remarkable agreement with the target.

The advantages of the proposed approach include:
 +\,no need to store past values of the sampled residual (even though the latter affects
 %influences the %evolution/
the  local increments in the
 %/growth of the
 error non-locally);
 +\,applicability to wide classes of time discretizations;
 +\,low regularity assumptions;
 %no corner compatibility conditions to be satisfied;
+\,the approach works seamlessly for arbitrarily large %terminal
times.\smallskip
%(--the error bound is fully computable, with all constants shown explicitly.)

%\newpage

\noindent{\it Notation.} We use the standard inner product $\langle\cdot,\cdot\rangle$ and the norm $\|\cdot\|$
in the space $L_2(\Omega)$, as well as the standard spaces %/norms in
$L_\infty(\Omega)$, $H^1_0(\Omega)$,
$L_{\infty}(0,t;\,L_2(\Omega))$, and $W^{1,\infty}(t',t'';\,L_2(\Omega))$.
The notation $v^+:=\max\{0,\,v\}$ is used for the positive part of a generic function $v$.%\vspace{-0.2cm}

\section{A posteriori error estimates in the $L_2(\Omega)$ norm}

Given a solution approximation $u_h$ such that $u_h=u$ for $t=0$ and on $\pt\Omega$, we shall use its %standard
residual
$R_h(\cdot,t):=(D_t^\alpha+\LL)u_h(\cdot,t)-f(\cdot,t)$ for $t>0$, as well as the operator %$(D_t^\alpha+\lambda)^{-1}$
%defined by
\beq\label{Dlammbda_inv}
(D_t^\alpha+\lambda)^{-1}v(t) :=
\int_0^t\! (t-s)^{\alpha-1}\,
E_{\alpha,\alpha}(-\lambda [t-s]^\alpha)\,v(s)\,ds
\quad\;\; \mbox{for}\;t>0.
\eeq
Here $E_{\alpha,\beta}(s)=\sum_{k=0}^\infty\{\Gamma(\alpha k+\beta)\}^{-1}s^k$
is a generalized Mittag-Leffler function.
A comparison with \eqref{CaputoEquiv} shows that  $(D_t^\alpha+0)^{-1}:=J_t^\alpha$.

\begin{remark}\label{rem_comparison} Note
\cite[Remark~7.1]{Diet10}, {\cite[(2.11)%; see alsot (2.19)
]{sakamoto}}
that
%It is clear from, a simplified version of \cite[(2.11); see alsot (2.19)]{sakamoto}
\eqref{Dlammbda_inv} gives a solution of the equation
$(D_t^\alpha+\lambda)w(t)=v(t)$ for $t>0$ subject to $w(0)=0$.
Also, $E_{\alpha,\alpha}$ in \eqref{Dlammbda_inv} is positive $\forall\,\lambda\in\R$
%\cite{miller_samko}
\cite[Lemma~3.3]{sakamoto}; hence, $v\ge 0$ implies $w\ge 0$.
\end{remark}

The main results of the paper are as follows.

\begin{theorem}\label{the_L2}
Let $\LL$ in \eqref{problem}, for some $\lambda\in\R $, satisfy $\langle \LL v,v\rangle\ge \lambda\|v\|^2$ $\forall\,v\in H_0^1(\Omega)$.
Suppose a unique solution $u$ of \eqref{problem} and its approximation $u_h$ are
in $L_{\infty}(0,t;\,L_2(\Omega)) \cap W^{1,\infty}(\epsilon,t;\,L_2(\Omega))$ for any $0<\epsilon<t\le T$,
and also in
$H^1_0(\Omega)$ for any $t>0$,
while
$u_h(\cdot,0)= u_0$ and
$R_h(\cdot,t)=(D_t^\alpha+\LL)u_h(\cdot,t)-f(\cdot,t)$.
Then %for the error of the latter one has
\beq\label{L2_error}
\|(u_h-u)(\cdot,t)\|\le (D_t^\alpha+\lambda)^{-1}\|R_h(\cdot, t)\|\qquad\mbox{for}\;\;t>0.
%\le J_t^{\alpha}\|R_h(\cdot, t)\|_{L_2(\Omega)},
\eeq
%where $R_h(\cdot,t)=(D_t^\alpha+\LL)u_h(\cdot,t)-f(\cdot,t)$.
\end{theorem}

\begin{corollary}%[residual barrier]
\label{cor1_L2}
Under the conditions of Theorem~\ref{the_L2},
if $\|R_h(\cdot,t)\|\le %\RR(t):
(D_t^\alpha+\lambda)\EE(t)$ $\forall\,t>0$ for some
barrier function $\EE(t)\ge 0$ $\forall\,t\ge0$, then $\|(u_h-u)(\cdot,t)\|\le \EE(t)$ $\forall\,t\ge0$.
\end{corollary}

The above corollary may seem
to imply that one can get any desirable
 pointwise-in-time error profile $\EE(t)$ on demand.
The tricky part is to ensure that $(D_t^\alpha+\lambda)\EE(t)>0$ for $t>0$, which is not true for
a general positive $\EE$. %This issue is addressed by the following result.
Two possible error profiles will be described by the following result.

\begin{corollary}\label{cor1_L2_bounds}
Under the conditions of Theorem~\ref{the_L2} with $\lambda\ge0$,
for the error $e=u_h-u$
 one has
\begin{subequations}\label{barrier_bounds}
\begin{align}
%\|e(\cdot, t)\|\le \sup_{s>0}\Bigl\{(...)^{-1}\|R(\cdot,s)\|\Bigr\}
&\|e(\cdot,t)\|\le \sup_{0<s\le t}\!\left\{ \frac{ \|R_h(\cdot,s)\|}{\RR_0(s) }\right\},
&&
\|e(\cdot,t)\|\le t^{\alpha-1}\!
\sup_{0<s\le t}\!\left\{ \frac{ \|R_h(\cdot,s)\|}{\RR_1(s) }\right\},
%\sup_{0<s\le t}\Bigl\{%\RR_2(s)^{-1}
%\|R_h(\cdot,s)\|/\RR_1(s)\Bigr\},
\label{barrier_bounds_a}
\\[0.3cm]
&\RR_0(t):=\{\Gamma(1-\alpha)\}^{-1}\,t^{-\alpha}+\lambda,
&&\RR_1(t):=\{\Gamma(1-\alpha)\}^{-1}\,t^{-1}\rho(\tau/t)+\lambda\,\EE_1(t),
\end{align}
\beq
\EE_1(t):=\max\{\tau, t\}^{\alpha-1},\quad
\rho(s):=s^{-\beta}[1-((1-s)^+)^\beta]\ge s^{-\beta}\min\{\beta s,\,1\},
\quad\beta:=1-\alpha,
%\quad\tau>0,
\eeq
where $\tau>0$ is an arbitrary parameter (and $t^{\alpha-1}$ in \eqref{barrier_bounds_a} can be replaced by $\EE_1(t)$).
\end{subequations}
\end{corollary}

\begin{corollary}[$u_h(\cdot,0)\neq u_0$]\label{cor_u0}
Suppose that $u_h$ is continuous in $t$ for $t\ge 0$ and does not satisfy $u_h(\cdot,0)= u_0$.
Then  Theorem~\ref{the_L2} and Corollaries~\ref{cor1_L2} and~\ref{cor1_L2_bounds}
are valid
with $R_h(\cdot,t)=[u_h(\cdot, 0)-u_0]\{\Gamma(1-\alpha)\}^{-1}t^{-\alpha}+(D_t^\alpha+\LL)u_h(\cdot,t)-f(\cdot,t)$.
\end{corollary}

\begin{remark}[$\RR_0$ v $\RR_1$]\label{rem_R_01}
If  uniform-in-time accuracy is targeted, then
the first bound in \eqref{barrier_bounds_a}, with the residual barrier $\RR_0$, is to be employed.
The second bound, with $\RR_1$, is less intuitive. It may be viewed as
an a posteriori analogue of
 pointwise-in-time a priori error bounds of type \cite[(3.2)]{NK_XM} and
 \cite[(4.2)]{NK_L2} on graded meshes $\{T(j/M)^r\}_{j=0}^M$.
 Let $q$ denote the order of the method (with $q=2$ for the L1 method).
 The latter bounds show (for three discretizations) that
 the error behaves as $t^{{\alpha-1}}M^{-r}$ %(compare with \eqref{barrier_bounds_a}!)
 for $1\le r\le q-\alpha$ (with a logarithmic factor for $r=q-\alpha$),
while the optimal convergence rate  $q-\alpha$ in positive time is attained
 if $r\approx q-\alpha$.
 Hence, it is reasonable to expect that an adaptive algorithm using
%appropriately-normalized
residual barriers $\RR_0$ and $\RR_1$
  will respectively yield optimal convergence rates $q-\alpha$ globally or in positive time.
  This agrees, and remarkably well, with
 the numerical results in \S\ref{sec4}  for the L1 method, and in \cite{SF_NK} for a number of higher-order methods.
\end{remark}

\begin{remark}[$\LL u_0\not\in L_2(\Omega)$]\label{rem_u0_nonreg}
If $u_0$ is not sufficiently smooth (see, e.g., test problem C in \S\ref{ssec_num_par}),
then (depending on the interpolation of $u_h$ in time) the residual $R_h(\cdot,t)$ on the first time interval $(0,t_1)$ may fail to be in $L_2(\Omega)$.
One way to rectify this is to reset $u_h(\cdot,t):=u_h(\cdot, t_1)$ for $t\in (0,t_1]$.
With this modification, all above results become applicable. Importantly, all changes in $u_h$ need to be reflected when computing its residual $R_h$;
in particular, as $u_h$ has been made discontinuous at $t=0$,
Corollary~\ref{cor_u0} is to be employed.
\end{remark}

The remainder of this section is devoted to the proofs of the above results.
The key role %in the analysis
will be played by the following auxiliary lemma,
a discrete version of which has been useful in the a priori error analysis; see, e.g., in \cite[(3.4)]{NK_MC_L1}.

 \begin{lemma}\label{lem_aux}
Suppose that $v(\cdot,0)=0$ and
$v\in L_{\infty}(0,t;\,L_2(\Omega)) \cap W^{1,\infty}(\epsilon,t;\,L_2(\Omega))$ for any $0<\epsilon<t\le T$.
%$v\in C[0,T]\cap  C^{0,1}(0,T]$
%{\color{red} where $C^{0,1}(0,T]$ is the H\"older space [to apply to  $u_h$]}
%
Then
$$
%v(t)\,
%\bigl (D_t^\alpha v(t)\bigr)\,v(t)  \ge \bigl( D_t^\alpha |v(t)|\bigr)\,|v(t)|
%\quad
\langle  D_t^\alpha v(\cdot,t),\,v(\cdot,t) \rangle \ge \bigl (D_t^\alpha \|v(\cdot,t)\|\bigr)\|v(\cdot,t)\|\qquad\mbox{for}\;\;t>0.
$$
\end{lemma}
%\newpage

\begin{proof}
In view of \eqref{CaputoEquiv}, replacing $\pt_s v(\cdot,s)$ in $D_t^\alpha v(\cdot,t)$ by $\pt_s\{v(\cdot,s)- v(\cdot, t)\}$ and then integrating by parts (with $v(\cdot,0)=0$), one gets
\beq\label{aux_lem_int}
\Gamma(1-\alpha)\,D_t^\alpha v(\cdot,t)%=\int_{0}^t(t-s)^{-\alpha}\, \pt_s[v(s)- v( t)]\, ds
=t^{-\alpha} v(\cdot,t)+\int_{0}^t\!\alpha(t-s)^{-\alpha-1}\, \{v(\cdot,t)- v(\cdot, s)\}\, ds.
\eeq
It remains to take the inner product of \eqref{aux_lem_int} with $v(\cdot,t)$. Then in the right-hand side $v(\cdot\,t)$ becomes $\|v(\cdot,t)\|^2$, while
$\langle v(\cdot,t)-v(\cdot,s),\,v(\cdot,t)\rangle\ge \{\|v(\cdot,t)\|-\|v(\cdot\,s)\|\}\|v(\cdot,t)\|$, so the desired assertion follows.
Note that the inner product of \eqref{aux_lem_int} with $v(\cdot,t)$ is well-defined, in view of $\|v(t)-v(s)\|\le C_t(t-s)$ for any fixed $t>0$ (with a $t$-dependent constant $C_t$).
Similarly, a version of
\eqref{aux_lem_int} for
$D_t^\alpha \|v(\cdot,t)\|$
remains well-defined
%if $\{v(t)-v(s)\}$    is replaced by $\{|v(t)|-|v(s)|\}$
%, the resulting integral remains well-defined, as
as $|\,\|v(\cdot,t)\|-\|v(\cdot,s)\|\,|\le \|v(\cdot,t)-v(\cdot,s)\|$.
%So one also gets a version of \eqref{aux_lem_int} for $D_t^\alpha |v|(t)$.
%
\end{proof}

\begin{remark}\label{rem_caputo_alt_def}
One may consider \eqref{aux_lem_int} an alternative definition of $D_t^\alpha$ (with an obvious modification for the case $v(\cdot,0)\neq 0$; see also \cite{Bunner_etal_2015}),
which can be applied to less smooth functions, including functions discontinuous at $t=0$.
Consider $\EE_0(t):=1$ for $t>0$ with $\EE_0(0):=0$.
Then, a %straightforward
calculation using~\eqref{aux_lem_int} yields $\Gamma(1-\alpha)\,D_t^\alpha \EE_0(t)=t^{-\alpha}$.
The same result may be obtained using the original definition \eqref{CaputoEquiv} combined with $\pt_t \EE_0(t)=\delta(t)$, the Dirac delta-function, or representing $\EE_0$ as the limit of a sequence of continuous piecewise-linear functions
(similarly to \cite[Remark 2.4]{NK_TL_sinum}).
\end{remark}

\noindent{\it Proof of Theorem~\ref{the_L2}.}~
Set $e:=u_h-u$. Then $e(\cdot,0)=0$ and
$(D_t^\alpha+\LL) e(\cdot,t)=R_h(\cdot,t)$ for $t>0$ subject to
$e=0$ on $\pt\Omega$. Taking the inner product of this equation with $e(\cdot,t)$, then applying Lemma~\ref{lem_aux} and
$\langle \LL e,e\rangle\ge \lambda\|e\|^2$,
one arrives at
\beq\label{e_eq}
(D_t^\alpha+\lambda) \|e(\cdot,t)\|\le \|R_h(\cdot,t)\| \qquad\mbox{for}\;\; t>0.
\eeq
Now, in view of Remark~\ref{rem_comparison},
 $(D_t^\alpha+\lambda)\bigl\{(D_t^\alpha+\lambda)^{-1}\|R_h(\cdot,t)\|-\|e(\cdot,t)\|\bigr\}\ge 0$ yields
 the desired bound \eqref{L2_error}.
\hfill$\square$
\bigskip

\noindent{\it Proof of Corollary~\ref{cor1_L2}.}~
First, suppose that $\EE(0)=0$. Then,  by \eqref{e_eq} combined with the corollary hypothesis,
$(D_t^\alpha+\lambda)(\EE(t)-\|e(\cdot,t)\|)\ge 0$ subject to $\EE(0)-\|e(\cdot,0)\|=0$.
In view of Remark~\ref{rem_comparison}, this
immediately yields the desired assertion $\EE(t)-\|e(\cdot,t)\|\ge 0$.
Otherwise, if $\EE(0)>0$, then $\EE(t)-\|e(\cdot,t)\|$ will include an additional positive component
$\EE(0)\, E_{\alpha,1}(-\lambda t^\alpha)$, so
$\EE(t)-\|e(\cdot,t)\|$ will remain positive.
\hfill$\square$
\bigskip

\noindent{\it Proof of Corollary~\ref{cor1_L2_bounds}.}~
As all operators are linear, it suffices to prove \eqref{barrier_bounds} with the $\sup\{\cdot\}$ terms equal to $1$, i.e. for $\|R_h(\cdot,t)\|\le \RR_p(t)$, $p=0,1$.

For the first bound in \eqref{barrier_bounds_a}, recall
from Remark~\ref{rem_caputo_alt_def} that
 for the function $\EE_0(t):=1$ for $t>0$ with $\EE_0(0):=0$ one has
 $(D_t^\alpha+\lambda)\EE_0(t)=\RR_0(t)$.
So an application of Corollary~\ref{cor1_L2} with $\EE(t):=\EE_0(t)$ yields
the first desired bound  $\|e(\cdot,t)\|\le \EE_0(t)=1$ for $t>0$.

For the second bound in \eqref{barrier_bounds_a},
set $\EE_1(t):=\max\{\tau, t\}^{\alpha-1}$ for $t>0$ with $\EE_1(0):=0$
(a similar barrier was used in \cite[Appendix~A]{NK_MC_L1}, \cite[Lemma 2.3]{NK_XM}).
Now it suffices to check that $ D_t^\alpha\EE_1(t)=\{\Gamma(1-\alpha)\}^{-1}\,t^{-1}\rho(\tau/t)$, as then
$(D_t^\alpha+\lambda)\EE_1(t)= \RR_1(t)\ge \|R_h\|$,
so an application of Corollary~\ref{cor1_L2} immediately
yields the desired bound
$\|e(\cdot,t)\|\le \EE_1(t)\le t^{\alpha-1}$.

To evaluate $ D_t^\alpha\EE_1(t)$,
set %$\color{magenta}\beta:=1-\alpha$ and
$\hat\tau:=\tau/t$, and
note that
$\EE_1(t)=\tau^{-\beta}\EE_0(t) - (\tau^{-\beta}-t^{-\beta})^+$.
Then for $t\le \tau$, i.e. $\hat\tau\ge 1$, one has
$\EE_1(t):=\tau^{-\beta}\EE_0(t)$, so $\Gamma(1-\alpha) D_t^\alpha \EE_1(t)=\tau^{-\beta}t^{-\alpha}=t^{-1}\hat\tau^{-\beta}=t^{-1}\rho(\hat\tau)$
as required.
For $t> \tau$, i.e. $\hat\tau\in(0,1)$, note that
$\pt_s (\tau^{-\beta}-s^{-\beta})^+ =-\pt_s(s^{-\beta})
=\beta s^{-\beta-1} $, so\vspace{-0.15cm}
$$
\Gamma(1-\alpha)\,D^\alpha_t \EE_1(t)=t^{-1}\hat\tau^{-\beta}
-\beta\int_{\tau}^t\! s^{-\beta-1}(t-s)^{-\alpha}\,ds
=t^{-1}\hat\tau^{-\beta}\bigl[1-(1-\hat\tau)^\beta\bigr],\vspace{-0.15cm}
$$
so we again get $\Gamma(1-\alpha)\,D^\alpha_t \EE_1(t)=t^{-1}\rho(\hat\tau)$. So
indeed, $  D_t^\alpha\EE_1(t)=\{\Gamma(1-\alpha)\}^{-1}\,t^{-1}\rho(\tau/t)$ for any $t>0$, as required.
\hfill$\square$\bigskip

\noindent{\it Proof of Corollary~\ref{cor_u0}.}~
Theorem~\ref{the_L2} and its two corollaries immediately apply to
 $u_h$ once it is reset to $u_0$ at $t=0$
(after which, it is worth noting, $u_h$ becomes right-discontinuous at $t=0$).
However, this modification of $u_h$ needs to be reflected in the computation of the residual  $R_h$ as follows.
Given $u_h$, continuous in $t$ for $t\ge 0$,
set $\bar u_h:=u_h$, and then reset $u_h(\cdot,0):=u_0$ (so $\bar u_h$ is continuous for $t\ge 0$, while $u_h$ is continuous and equal to $\bar u_h$ for $t>0$).
Now,
the residual $R_h$ of $u_h$ for $t>0$ is computed using
$u_h=[u_h(\cdot, 0^+)-u_0](\EE_0-1)+\bar u_h$,
so
 $D_t^\alpha u_h=[u_h(\cdot, 0^+)-u_0]D_t^\alpha \EE_0+D_t^\alpha \bar u_h$, where
 $D_t^\alpha \EE_0=\{\Gamma(1-\alpha)\}^{-1}t^{-\alpha}$; see
 Remark~\ref{rem_caputo_alt_def}.
 In other words, to compensate for $u_h(\cdot,0^+)\neq u_0$, one needs to add $[u_h(\cdot, 0^+)-u_0]\{\Gamma(1-\alpha)\}^{-1}t^{-\alpha}$
to $R_h$ of Theorem~\ref{the_L2}.
\hfill$\square$
%\vspace{-0.3cm}

 \section{Generalization for the $L_\infty(\Omega)$ norm}

 Let %the spatial operator $\LL$  be
 $\LL u := \sum_{k=1}^d \Bigl\{-%\pt_{x_k}\!(
 a_k(x)\,\pt^2_{x_k}\!u%)
 + b_k(x)\, \pt_{x_k}\!u \Bigr\}+c(x)\, u$
 in \eqref{problem},
with sufficiently smooth coefficients $\{a_k\}$, $\{b_k\}$ and $c$ in $C(\bar\Omega)$, for which we assume that $a_k>0$ in $\bar\Omega$,
and also $c\ge \lambda\ge 0$ (while $\langle \LL v,v\rangle\ge \lambda\|v\|^2$ is not required in this section).

\begin{lemma}[maximum/comparison principle]\label{lem_max_pr}
Suppose that $v(x,t)\ge 0$ for $t=0$ and $x\in\pt\Omega$,
and $v$ is in
$ C(\bar\Omega\times[0,t]) \cap W^{1,\infty}(\epsilon,t;\,L_\infty(\Omega))$ for any $0<\epsilon<t\le T$
and also in
$C^2(\Omega)$ for any $t>0$.
Then $(D_t^\alpha+\LL)v\ge 0$ in $(0,T]\times\Omega$ implies $v\ge 0$ in $[0,T]\times\bar\Omega$.
\end{lemma}

\begin{proof}
 This result is given in \cite[Theorem~2]{luchko} under a stronger condition  that $v(x,\cdot)\in C^1(0,T]\cap W^{1,1}(0,T)$.
 An inspection of the proof shows that this condition is only required to apply \cite[Theorem~1]{luchko}
 (the maximum principle for $D_t^\alpha$).
 The proof of the latter  relies on the representation of type \eqref{aux_lem_int}
 and remains valid under our weaker assumptions.
 (A~similar, but not identical, result is also  given in \cite[Theorem~4.1]{Bunner_etal_2015}.)
\end{proof}

\begin{theorem}
 Under the above assumptions on $\LL$,
 let a unique solution $u$ of \eqref{problem} and its approximation $u_h$ be
in $C(\bar\Omega\times[0,t]) \cap W^{1,\infty}(\epsilon,t;\,L_\infty(\Omega))$ for any $0<\epsilon<t\le T$,
and also in
$C^2(\Omega)$ for any $t>0$.
Then the error bounds of Theorem~\ref{the_L2} %(bound \eqref{L2_error})
and Corollaries~\ref{cor1_L2} and~\ref{cor1_L2_bounds} remain true with $\|\cdot\|=\|\cdot\|_{L_2(\Omega)}$
replaced by $\|\cdot\|_{L_\infty(\Omega)}$.
\end{theorem}

\begin{proof}
 We shall start with  Corollary~\ref{cor1_L2}.
It is now assumed that $\|R_h(\cdot,t)\|_{L_\infty(\Omega)}\le
(D_t^\alpha+\lambda)\EE(t)$ $\forall\,t>0$.
Noting that $R_h=(D_t^\alpha+\LL)e$ and $(D_t^\alpha+\lambda)\EE(t)\le (D_t^\alpha+\LL)\EE(t)$,
one concludes that $|(D_t^\alpha+\LL)e(x,t)|\le (D_t^\alpha+\LL)\EE(t)$ $\forall\,x\in\Omega,\,t>0$.
So an application of Lemma~\ref{lem_max_pr} yields the desired bound
 $|e(x,t)|\le \EE(t)$ in $(0,T]\times\Omega$.

The remaining statements follow from this version of Corollary~\ref{cor1_L2}; to be more precise,
 the new version of Theorem~\ref{the_L2}  is obtained using $\EE(t):=(D_t^\alpha+\lambda)^{-1}\|R_h(\cdot,t)\|_{L_\infty(\Omega)}$,
and the new version of Corollary~\ref{cor1_L2_bounds} using $\EE(t):=\EE_p(t)$, $p=0,1$.
\end{proof}%\vspace{-0.3cm}

\section{Application for the L1 method}\label{sec4}

Given an arbitrary temporal mesh $\{t_j\}_{j=0}^M$ on $[0,T]$, let $\{u_h^j\}_{j=0}^M$ be the semi-discrete approximation for \eqref{problem} obtained using the
popular L1 method \cite{laz_review,stynes_L1}.
Then its standard Lagrange piecewise-linear-in-time interpolant $u_h$, defined on $\bar\Omega\times[0,T]$,
 satisfies\vspace{-0.1cm}
\beq\label{L1method}
(D_t^\alpha+\LL)u_h(x, t_j)=f(x, t_j)\qquad \mbox{for}\;\; x\in\Omega,\;\;  j=1\ldots, M,\vspace{-0.1cm}
\eeq
subject to $u_h^0:=u_0$ and $u_h=0$ on $\pt \Omega$.

%The above immediately implies that
So for
the residual of $u_h$ one immediately gets $R_h(\cdot, t_j)=0$ for $j\ge 1$,
i.e. on each $(t_{j-1},t_j)$ for $j>1$, the residual is a non-symmetric bubble.
Hence, for the piecewise-linear interpolant $R_h^I$ of  $R_h$ one has $R_h^I=0$ for $t\ge t_1$,
and, more generally,
$R_h^I=[\LL u_0-f(\cdot, 0)](1-t/t_1)^+$ for $t>0$
(where we used $R_h(\cdot,0)=\LL u_0-f(\cdot,0)$, in view of $D_t^\alpha u_h^0(\cdot, 0)=0$).
Finally, note that  $R_h-R_h^I=(D_t^\alpha u_h-f)-(D_t^\alpha u_h-f)^I$,
(in view of
$(\LL u_h)^I%=\LL (u_h)^I
=\LL u_h$).
%Finally,
%$R_h=(D_t^\alpha u_h-f)-(D_t^\alpha u_h-f)^I+(\LL u_0-f(\cdot, 0))(1-t/t_1)^+$ for $t>0$
%
In other words, one can compute $R_h$ by sampling, using parallel/vector evaluations, without a direct application of $\LL$ to $\{u^j_h\}$.%
%\vspace{-0.3cm}

\subsection{Numerical results for a test without spatial derivatives}

\noindent{\it Test problem A.}
We start our numerical experiments with a version of \eqref{problem} without spatial derivatives, with $\LL:=3$
and the exact solution $u=u(t)=t^\alpha-t^2$ (which exhibits a typical singularity at $t=0$) for $t\in(0,1]$.
For this problem,
a straightforward adaptive algorithm (see \S\ref{sec_appA})
was employed,
motivated by \eqref{barrier_bounds}, and so %aiming to
constructing a temporal mesh such that $\|R_h(\cdot,t)\|\le \TOL\cdot\RR_p(t)$, $p=0,1$, with  $\tau:=t_1$ in $\RR_1$.
%(For the current test, $\|\cdot\|$ becomes the absolute value.)

  \begin{figure}[t!]%[b!]
% \vspace*{-0.5pc}
\begin{center}
\hspace*{-0.3cm}\includegraphics[height=0.28\textwidth]{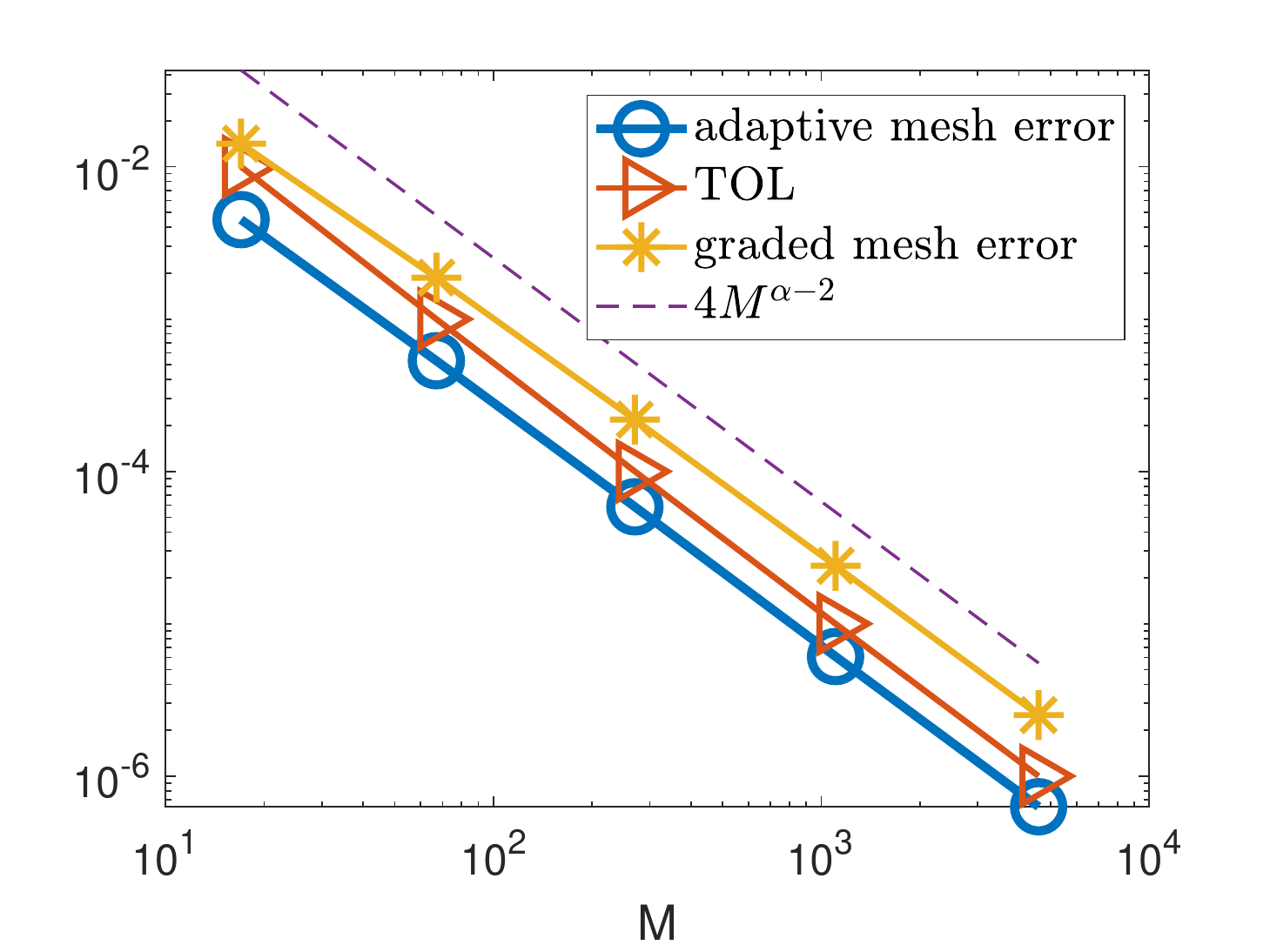}\hspace*{-0.45cm}\includegraphics[height=0.28\textwidth]{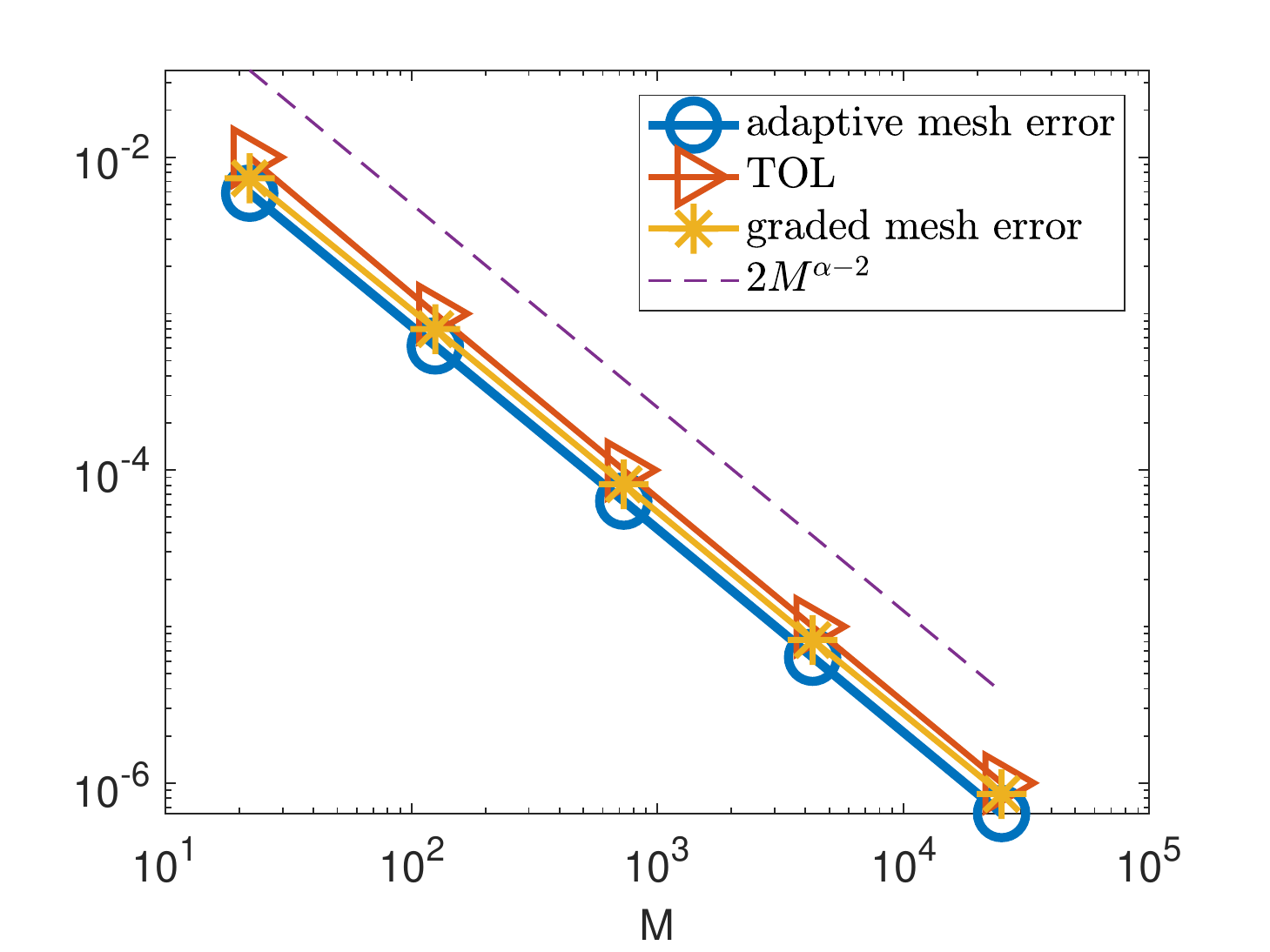}%
\hspace*{-0.45cm}\includegraphics[height=0.28\textwidth]{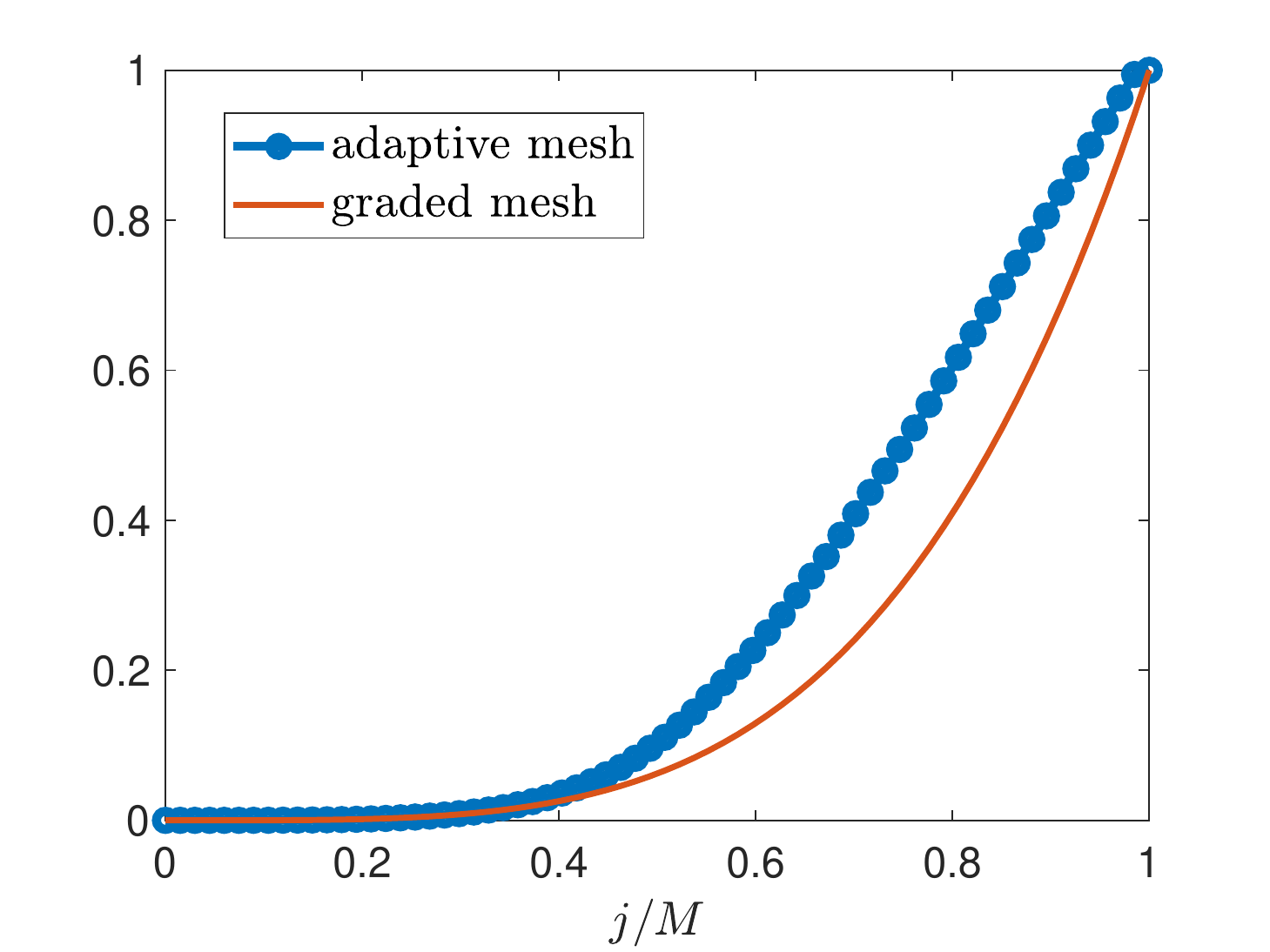}\hspace*{-0.45cm}
%{paradigm_alph04_lambda3_bw}
%\\%
%\includegraphics[height=0.30\textwidth]{r__alpha_L1}~~\includegraphics[height=0.30\textwidth]{r__04_L1}%
\end{center}
\vspace{-0.6cm}
 \caption{\label{fig_ODE_graded}\it\small
Adaptive algorithm with $\RR_0(t)$ for test problem A: %loglog graphs of the
$\max_{[0,T]}|e(t)|$ on the adaptive mesh, the corresponding $\TOL$ and error on the graded mesh,
%with $r=(2-\alpha)/\alpha$,
%, and the loglog line $M^{2-\alpha}$
$\alpha=0.4$ (left) and $\alpha=0.7$ (centre).
Right: graphs of $\{t_j\}_{j=0}^M$ as a function of $j/M$ for
the adaptive mesh v graded mesh with $r=(2-\alpha)/\alpha$,
%represented as mesh-generating functions $t_j=$
 $\alpha=0.7$, $\TOL =10^{-3}$, $M=67$.}
 \end{figure}

   \begin{figure}[t!]%[b!]
% \vspace*{-0.5pc}
\begin{center}
\hspace*{-0.3cm}\includegraphics[height=0.28\textwidth]{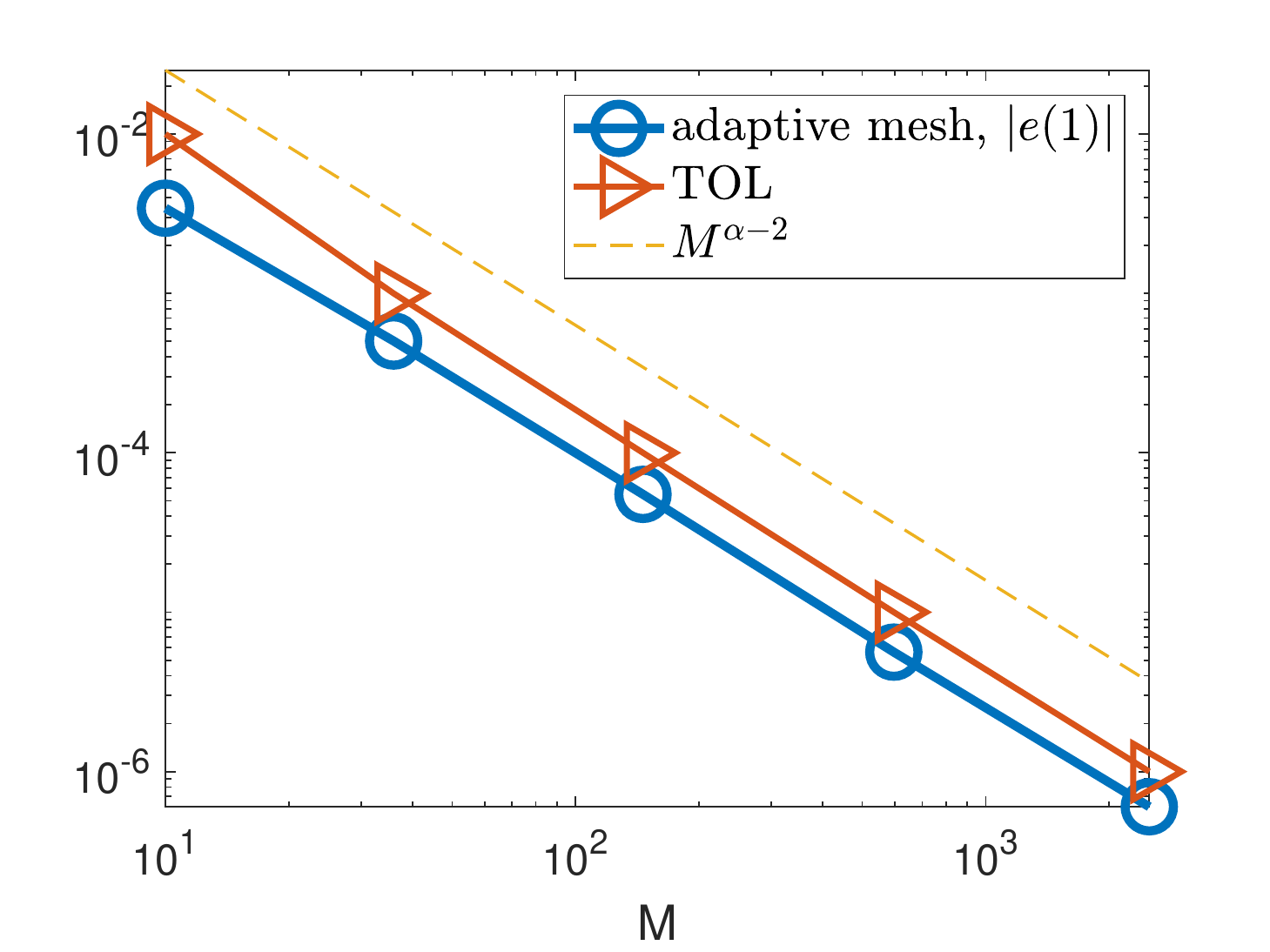}\hspace*{-0.45cm}\includegraphics[height=0.28\textwidth]{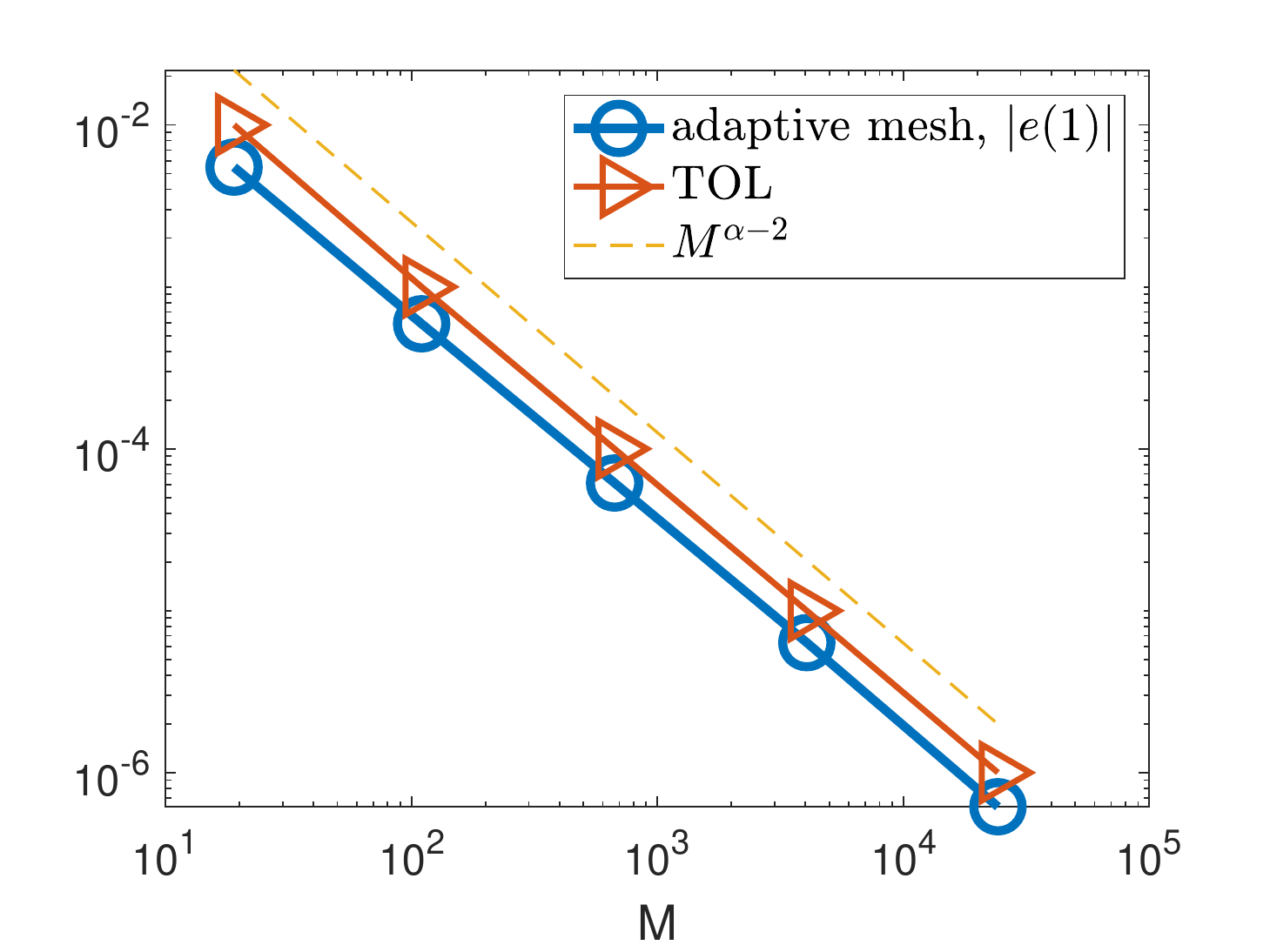}%
\hspace*{-0.45cm}\includegraphics[height=0.28\textwidth]{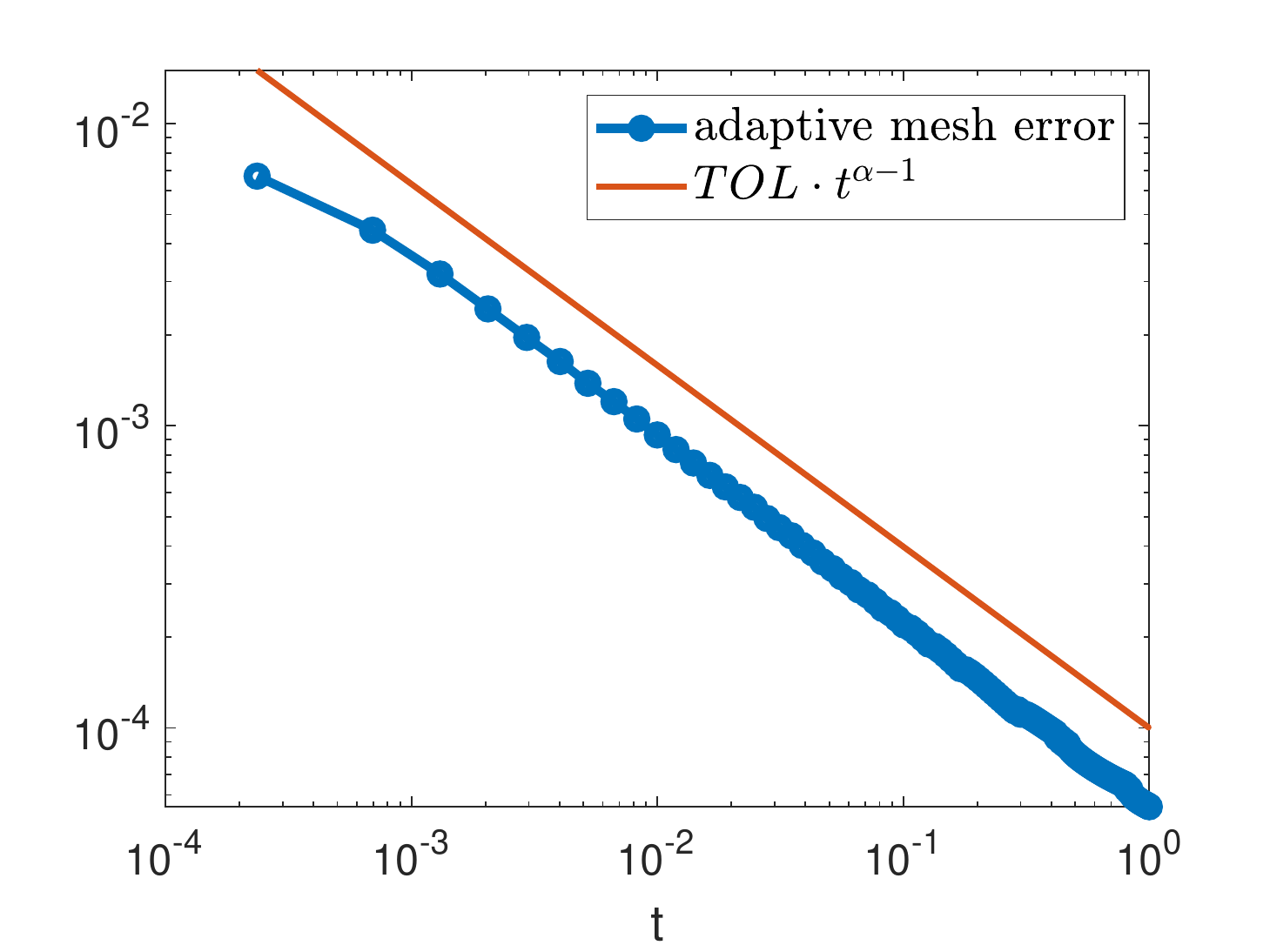}\hspace*{-0.45cm}
%{paradigm_alph04_lambda3_bw}
%\\%
%\includegraphics[height=0.30\textwidth]{r__alpha_L1}~~\includegraphics[height=0.30\textwidth]{r__04_L1}%
\end{center}
\vspace{-0.6cm}
 \caption{\label{fig_ODE_uni}\it\small
 Adaptive algorithm with $\RR_1(t)$ for test problem A: %loglog graphs of the
$|e(1)|$ on the adaptive mesh and the corresponding $\TOL$,
%with $r=(2-\alpha)/\alpha$,
%, and the loglog line $M^{2-\alpha}$
$\alpha=0.4$ (left) and $\alpha=0.7$ (centre).
Right:
log-log graph of the pointwise error $|e(t_j)|$ on the adaptive mesh v $\TOL\cdot t^{\alpha-1}$,
%represented as mesh-generating functions $t_j=$
 $\alpha=0.4$, $\TOL =10^{-4}$, $M=146$.}
 \end{figure}

The errors and rates of convergence obtained using residual barriers $\RR_0(t)$ and $\RR_1(t)$ are presented in Fig.\,\ref{fig_ODE_graded} \& \ref{fig_ODE_uni}.  %respectively.
For $\RR_0$, the errors on the adaptive meshes were compared with the errors on the optimal graded meshes $\{t_j=T(j/M)^r\}_{j=0}^M$
with $r=(2-\alpha)/\alpha$
 \cite{NK_MC_L1,NK_XM,stynes_etal_sinum17}
for the same values of $M$. We observe that in both cases the optimal global rates of convergence $2-\alpha$
are attained. Furthermore, not only the adaptive meshes
successfully detect the solution singularity, but they
slightly outperform the optimal graded meshes.
%\newpage
%\noindent
For $\RR_1$, %a similar comparison was performed with uniform meshes, for which the theoretical pointwise error bound is $O(M^{-1}t^{\alpha-1})$
we observe the optimal rates of convergence $2-\alpha$ at terminal time $t=1$, which is consistent with the error bound
  \cite[(3.2)]{NK_XM} for a mildly graded mesh (see Remark~\ref{rem_R_01}).

\subsection{Numerical results for fractional parabolic test problems}\label{ssec_num_par}

\noindent{\it Test problem B.}
Next, we consider  \eqref{problem} for $(x,t)\in(0,\pi)\times(0,1]$
with $\LL=-\pt_x^2$
and the exact solution $u:=(t^{\alpha}-t^2)\,\sin(x^2/\pi)$, so we set $\lambda:=1$.
The same  adaptive algorithm %(see {\color{red}Appendix })
was employed with $\RR_0(t)$ from
 \eqref{barrier_bounds} to generate temporal meshes, while in space the problem was discretized
 on the uniform mesh with $10^4$ intervals using standard finite differences (equivalent to lumped-mass linear finite elements).
 The numerical results are given on Fig.\,\ref{fig_parabolic} (left, centre) are similar to those on Fig.\,\ref{fig_ODE_graded} for test problem A.
\medskip

\noindent{\it Test problem C.}
Our final test is \eqref{problem} for $(x,t)\in(0,\pi)\times(0,0.2]$
with $\LL=-\pt_x^2$, so %we set
$\lambda:=1$. Now
%$u_0:=\frac12\exp(-2x)$ is piecewise-linear
$u_0:=x$ for $x\le 1$ and
$u_0:=1-(x-1)/(\pi-1)$  for $x\ge 1$,
while $f:=0$.
%, so the exact solution is unknown.
 As $\LL u_0\not\in L_2(\Omega)$, to be able to compute $\|R_h\|$ on $(0,t_1)$, we change the interpolation of the computed solution
 $\{u_h^j\}_{j=0}^M$ on $(0,t_1]$ to piecewise-constant, as described in Remark~\ref{rem_u0_nonreg}.
 The residual becomes
 $R_h=[u_h^1-u_h^0]\{\Gamma(1-\alpha)\}^{-1}t^{-\alpha}+\LL u_h^1-f(\cdot,t)$ for $t\in(0,t_1]$
 and
  $R_h=(D_t^\alpha u_h-f)-(D_t^\alpha u_h-f)^I+[u_h^1-u_h^0]\{\Gamma(1-\alpha)\}^{-1}\sigma(t)$ for $t>t_1$,
  where $\sigma(t):=t^{-\alpha}-(1-\alpha)^{-1}[t^{1-\alpha}-(t-t_1)^{1-\alpha}]/t_1$.
 A fixed mesh with $10^5$ subintervals was used in space.
 The reference solution was computed on a finer mesh.
 %by dividing each spatial interval into 2 and each temporal interval into 4.
 The numerical results, given on Fig.\,\ref{fig_parabolic} (right) indicate that our adaptive algorithm provides adequate error control
 %for corner singularities
 for piecewise-linear initial data, as well as for more typical solution singularities at initial time.
 For a further numerical study of this approach, we refer the reader to \cite{SF_NK}.\vspace{-0.2cm}

   \begin{figure}[t!]%[b!]
% \vspace*{-0.5pc}
\begin{center}
\hspace*{-0.3cm}\includegraphics[height=0.28\textwidth]{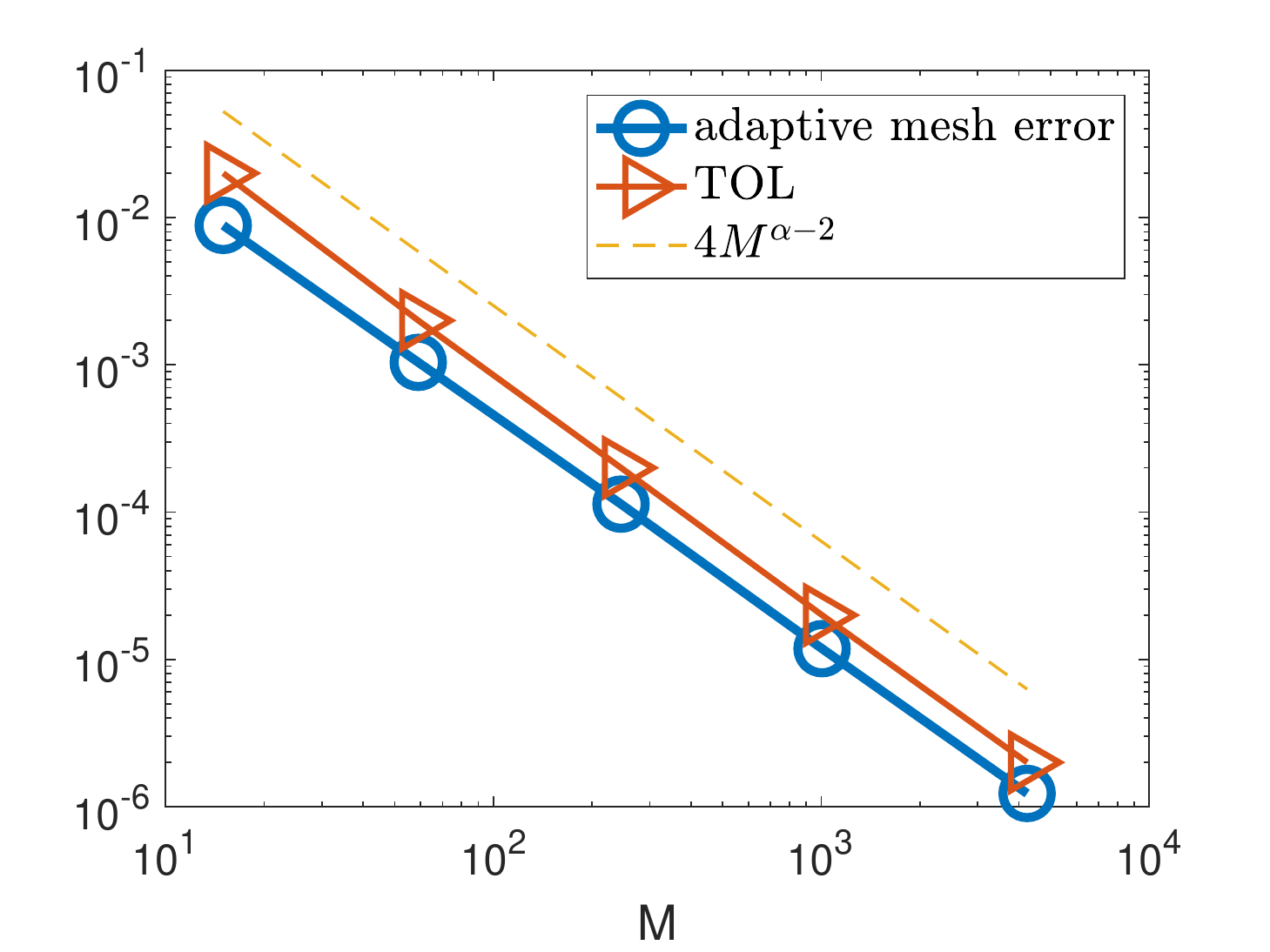}\hspace*{-0.45cm}\includegraphics[height=0.28\textwidth]{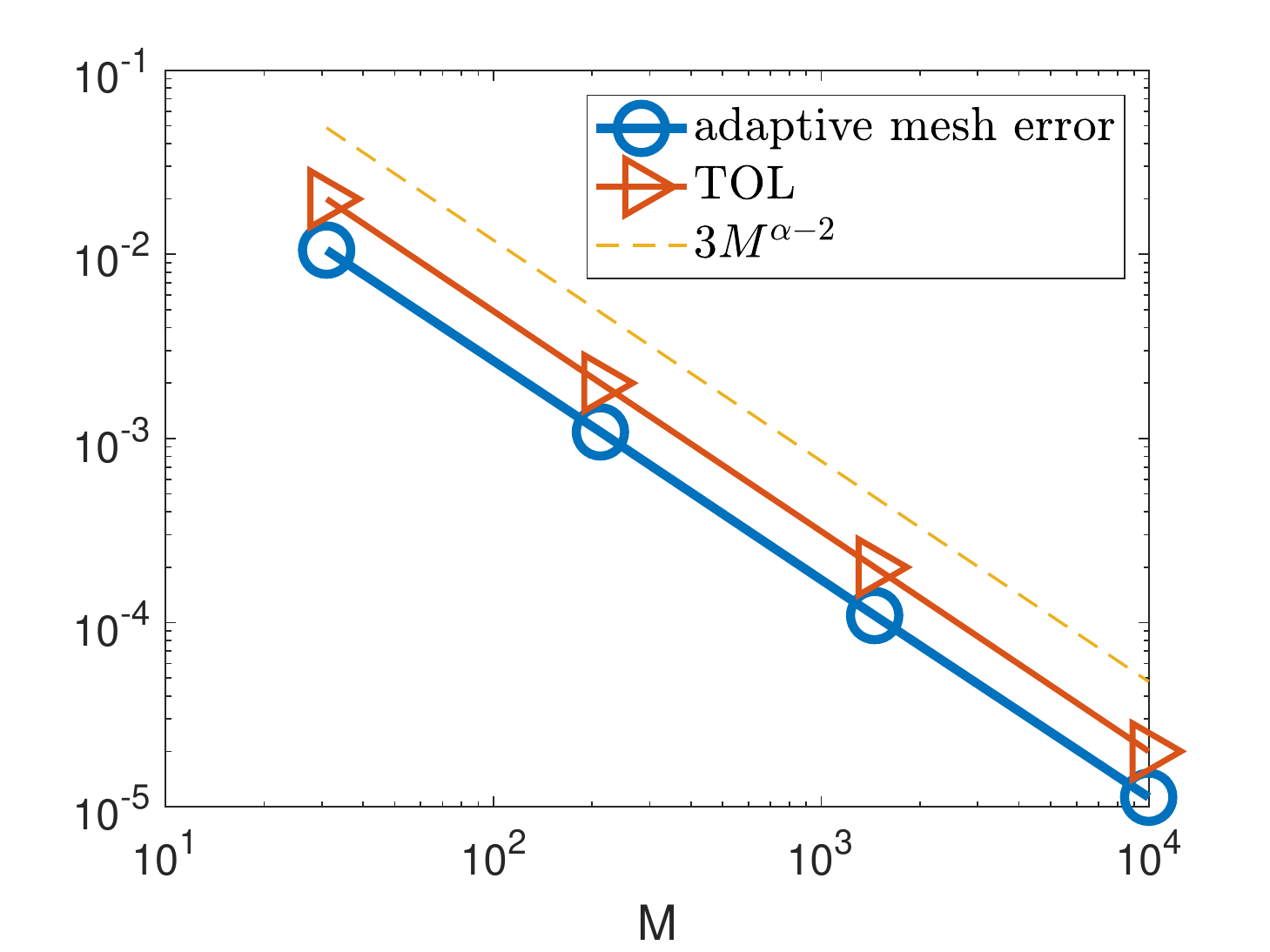}%
\hspace*{-0.45cm}\includegraphics[height=0.28\textwidth]{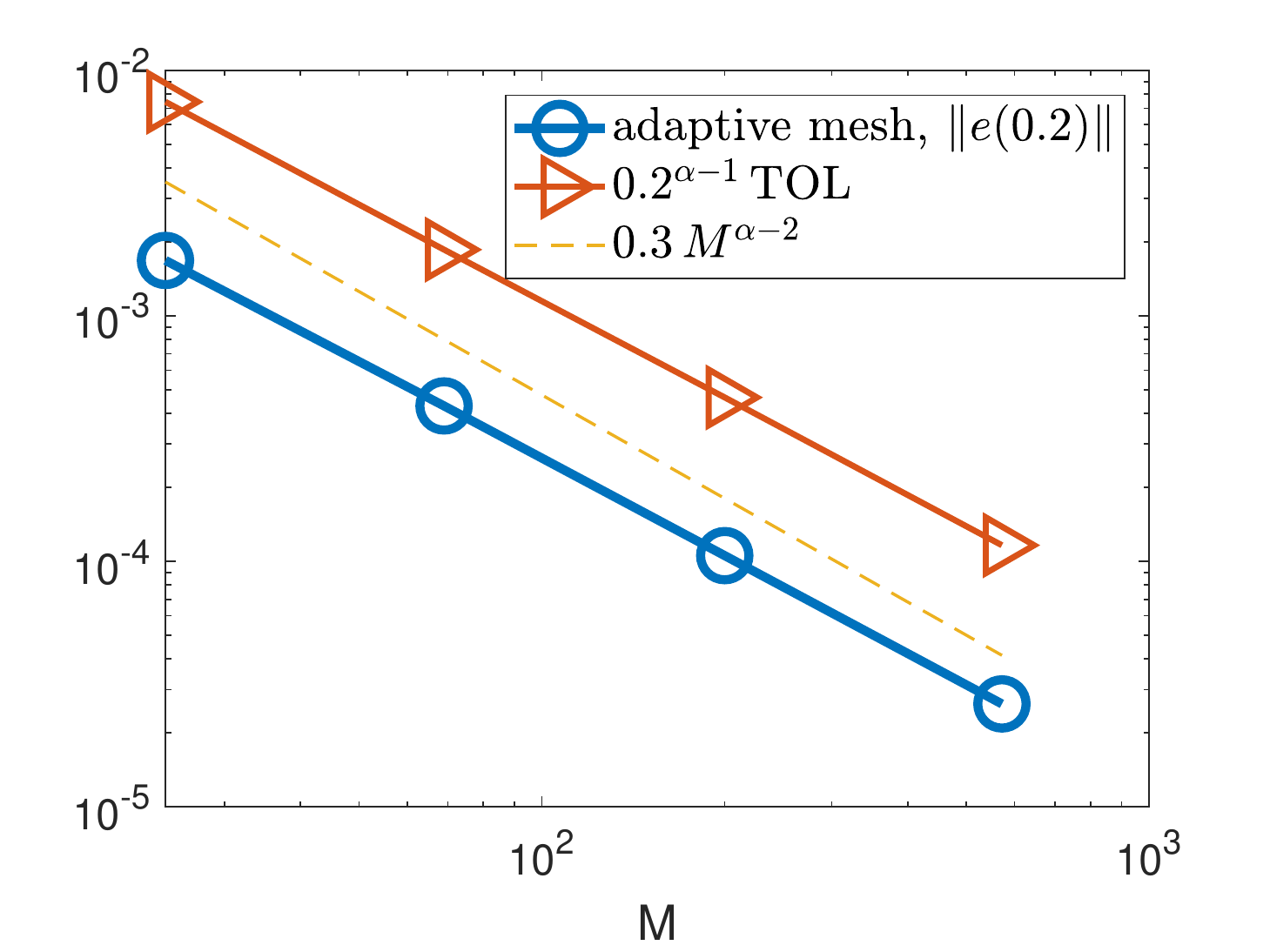}\hspace*{-0.45cm}
%{paradigm_alph04_lambda3_bw}
%\\%
%\includegraphics[height=0.30\textwidth]{r__alpha_L1}~~\includegraphics[height=0.30\textwidth]{r__04_L1}%
\end{center}
\vspace{-0.6cm}
 \caption{\label{fig_parabolic}\it\small
 Adaptive algorithm with $\RR_0(t)$ for parabolic test problem B:
 $\max_{t_j\in(0,T]}\|e(t_j)\|$ on the adaptive mesh
 and the corresponding $\TOL$, $\alpha=0.4$ (left) and $\alpha=0.8$ (centre).
 Adaptive algorithm with $\RR_0(t)$ for parabolic test problem C: , $\|e(0.2)\|$ and  $\TOL$, $\alpha=0.6$ (right).}
 \end{figure}

%\appendix
\subsection{Adaptive algorithm}\label{sec_appA}

\newcommand{\ind}{\hspace*{0.5cm}}

We employed the algorithm in Fig.\,\ref{fig_alg}, with parameters $Q:=1.1$, $\tau_{*}:=5\,\TOL^{1/\alpha}$ for $\RR_0$ and $\tau_{*}:=\TOL$ for $\RR_1$, $\tau_{**}:=0$.
%\vspace{0.4cm}
\noindent
Here we used the standard mathematical notation combined with
the MatLab {\tt while} loop syntax (where, to be precise, {\tt break} denotes an exit from the interior while loop).
%\bigskip

%\noindent~~~\hrulefill~~~

   \begin{figure}[t!]%[b!]
% \vspace*{-0.5pc}
\begin{center}
~\hfill\noindent%\fbox{~~~
    \parbox{0.7\textwidth}{%
    
    \noindent\hrulefill

        %\rule{0pt}{8pt}
        $u_h^0:=u_0$;\;\;$t_0:=0$;\;\;$t_1:=\min\{\tau_{*},\,T\}$;\;\;$m:=0$;

        {\tt while}\;\;$t_m<T$

        \ind $m:=m+1$;\;\;$flag :=0$;

        \ind {\tt while}\;\;$t_{m}-t_{m-1}>\tau_{**}$

        \ind\ind compute $u_h^m$ using \eqref{L1method}

        \ind\ind {\tt if}\;\;$\|R_h(\cdot,t)\|\le\TOL\cdot \RR_p(t)$ $\forall\,t\in(t_{m-1},t_m)$

        \ind\ind\ind {\tt if}\;\;$t_m=T$

        \ind\ind\ind\ind $M:=m$;\;\;{\tt break}

        \ind\ind\ind {\tt elseif}\;\;$t_m<T$

        \ind\ind\ind\ind  $\tilde u_h^m := u_h^m$;\;\;$\tilde t_m := t_m$;%\;\;%

        \ind\ind\ind\ind
        $t_m:=\min\{t_{m-1}+Q(t_m-t_{m-1}) ,\,T\}$;\;\;$flag: =1$;

        \ind\ind\ind {\tt end}

        \ind\ind {\tt else}%\;\;{\color{red}(X)} is not  satisfied

        \ind\ind\ind {\tt if}\;\;$flag =0$

        \ind\ind\ind\ind $t_m:=t_{m-1}+(t_m-t_{m-1})/Q$;

        \ind\ind\ind {\tt else}%\;\;$flag =1$

        \ind\ind\ind\ind $u_h^m := \tilde u_h^m$;\;\;$t_m := \tilde t_m$;%\;\;%

        \ind\ind\ind\ind
        $t_{m+1}:=\min\{t_m+(t_m-t_{m-1}),\, T\}$;\;\;{\tt break}

        \ind\ind\ind {\tt end}

        \ind\ind {\tt end}

        \ind {\tt end}

        {\tt end}%\\[-0.3cm]
        
        \noindent\hrulefill
%}
}\hfill~%
%
%\noindent~~~\hrulefill~~~
\end{center}
\vspace{-0.6cm}
 \caption{\label{fig_alg}\it\small
 Adaptive algorithm.}
 \end{figure}

\end{document}